\documentclass[12pt]{amsart}
\usepackage{amssymb}
\usepackage{amsmath, amscd}
\usepackage{amsthm}
\usepackage{xcolor}

\newtheorem{theorem}{Theorem}[section]

\newtheorem{proposition}[theorem]{Proposition}
\newtheorem{lemma}[theorem]{Lemma}

\newtheorem{corollary}[theorem]{Corollary}
\theoremstyle{definition}

\newtheorem{example}[theorem]{Example}

\newtheorem{remark}[theorem]{Remark}

\def\val#1{\vert #1 \vert}

\usepackage[utf8]{inputenc}
\usepackage{amssymb}
\usepackage{amsmath, amscd}
\usepackage{amsthm}
\usepackage{xcolor}

\begin{document}

\author[A.R. Chekhlov]{Andrey R. Chekhlov}
\address{Department of Mathematics and Mechanics, Regional Scientific and Educational Mathematical Center, Tomsk State University, 634050 Tomsk, Russia}
\email{cheklov@math.tsu.ru; a.r.che@yandex.ru}
\author[P.V. Danchev]{Peter V. Danchev}
\address{Institute of Mathematics and Informatics, Bulgarian Academy of Sciences, 1113 Sofia, Bulgaria}
\email{danchev@math.bas.bg; pvdanchev@yahoo.com}
\author[P.W. Keef]{Patrick W. Keef}
\address{Department of Mathematics, Whitman College, Walla Walla, WA 99362, USA}
\email{keef@whitman.edu}

\title[Semi-generalized Bassian groups] {Semi-generalized Bassian groups}

\maketitle

\begin{abstract} As a common non-trivial generalization of the concept of a proper generalized Bassian group, we introduce the notion of a {\it semi-generalized Bassian} group and initiate its comprehensive investigation. Precisely, we give a satisfactory characterization of these groups by showing in the cases of $p$-torsion groups,  torsion-free groups and splitting mixed groups their complete description. Moreover, we classify the groups with the clearly related property that every subgroup is essential in a direct summand of the whole group.
\end{abstract}

\vskip2.0pc

\section{Introduction and Motivation}
	
Throughout the rest of the paper, unless specified something else, all groups will be additively written Abelian groups. We will primarily use the notation and terminology of \cite{F0,F1,F2}, but we will follow somewhat those from \cite{K} and \cite{Gr} as well. We just recall that an arbitrary subgroup $H$ of a group $G$ is {\it essential} in $G$ if, for any non-zero subgroup $S$ of $G$, the intersection between $H$ and $S$ is also non-zero. It is an elementary exercise that any subgroup will be essential as a subgroup of itself and thus, in particular, the zero subgroup $\{0\}$ is essential in itself too.

We begin with a quick review of some of the most important notions which motivate our writing of the present article.

Imitating \cite{CDG1}, a group $G$ is said to be {\it Bassian} if the existence of an injective homomorphism $G \to G/N$, for some subgroup $N$ of $G$, forces that $N = \{0\}$. More generally, mimicking \cite{CDG2}, if this injection implies that $N$ is a direct summand of $G$, then $G$ is said to be {\it generalized Bassian}.

Note that Bassian groups were completely characterized in \cite{CDG1}. A critical example of a generalized Bassian group that is {\it not} Bassian is the so-called an infinite {\it elementary group} which is an arbitrary infinite direct sum of the cyclic group $\mathbb{Z}(p)$ for some fixed prime $p$. For our convenience, and to shorten the exposition, we shall hereafter denote by a {\it B+E-group} the direct sum of a Bassian group $B$ and an arbitrary elementary group $E$. It was proved in \cite{DK} that any generalized Bassian group necessarily has finite torsion-free rank as well as it was shown that a group is a B+E-group if, and only if, it is a subgroup of a generalized Bassian group. In particular, any generalized Bassian group is a B+E-group and the B+E-groups are closed under arbitrary subgroups. Likewise, it was conjectured there that every B+E-group is generalized Bassian. Notice also that it is still unsettled the problem of whether or {\it not} a subgroup of a generalized Bassian group is again so. However, for some other major properties of subgroups of generalized Bassian groups, we refer to \cite{DG} as well.

Taking into account all of the information we have so far, and in order to refine the group property of being generalized Bassian, we may define a group $G$ to be {\it semi-generalized Bassian} if, for any its subgroup $H$, the injective homomorphism $G \to G/H$ implies that $H$ is essential in a direct summand of $G$. Clearly, under our circumstances requested above, a Bassian group is semi-generalized Bassian and, even more generally, a generalized Bassian group is semi-generalized Bassian.

However, for an attractive exploration in-depth of these groups, it is needed to have at least one example of a semi-generalized Bassian group which is {\it not} generalized Bassian, calling it {\it proper} semi-generalized Bassian. In the next lines and also in what follows below, we demonstrate an abundance of such groups. In fact, the following two key points hold:

\medskip

$\bullet$ A routine check shows that each torsion-free divisible group of infinite rank will {\it not} be generalized Bassian, but it will satisfy the new definition of a semi-generalized Bassian group; in fact, the purification of any subgroup will be a direct summand containing the subgroup as an essential subgroup, as required.

\medskip

$\bullet$ A simple inspection illustrates that the direct sum of an infinite number of copies of the cyclic group $\mathbb{Z}(p^n)$ of order $n$ for some fixed natural number $n\geq 1$. In that case, again, any subgroup is an essential subgroup of a direct summand, so it trivially satisfies the new definition of a semi-generalized Bassian group and, if $n>1$, then it will surely {\it not} be generalized Bassian, as pursued.

\medskip

Our further work is organized as follows: In the present first section, we give a brief retrospection of the obtained by us results. In the next second section, we formulate our main results and give their complete proofs. Specifically, this paper's main goal is to characterize three classes of groups as follows:

\medskip

(1) The semi-generalized Bassian groups that are torsion-free.

(2) The semi-generalized Bassian groups that are torsion.

(3) The semi-generalized Bassian groups that are splitting mixed. 

(4) The groups with the clearly related property that every subgroup is essential in a direct summand of the group.

\medskip

Concretely, the most important of them state like this: {\it A direct summand of an arbitrary semi-generalized Bassian group is again so} (see Theorems~\ref{summands}). Moreover, excluding the general mixed case, we obtain the following complete characterizations of (reduced) semi-generalized Bassian groups: {\it A torsion-free group is semi-generalized Bassian if, and only if, it possesses finite rank or $G\cong D\oplus R$, where $D$ is divisible and $R$ is a finite-rank fully-decomposable homogeneous group} (see Theorem~\ref{5}); {\it A $p$-group $G$ is semi-generalized Bassian if, and only if, $G$ is divisible, or $G$ has finite $p$-rank, or for some positive integer $n$ we have $G\cong G'\oplus F$, where $G'$ is a group with $f_G'(\alpha)=0$ unless either $\alpha=n-1$ or $\alpha=n$, and $F$ is a finite subgroup such that $F[p]\subseteq p^{n+1} G[p]$} (see Theorem~\ref{main}); {\it The splitting group $G=T\oplus R$, where $R$ is the torsion-free part of $G$, is a semi-generalized Bassian group if, and only if, $T$ and $R$ are both semi-generalized Bassian groups and, moreover, if the rank of $R$ is infinite, then $T$ is divisible} (see Proposition~\ref{newone2}); {\it A reduced $p$-group $G$ has the property that every subgroup is an essential subgroup of a summand of $G$ if, and only if, $f_G(\alpha)=0$ unless either $\alpha=n-1$ or $\alpha=n$, and so any such a group is necessarily semi-generalized Bassian} (see Theorem~\ref{6}); {\it If $G$ is either a mixed or a torsion-free group, then every subgroup of $G$ is essential in some direct summand if, and only if, $G=D\oplus R$, where $D=T(D)\oplus D_0$ is a divisible group and $R$ is a homogeneous fully decomposable group of finite rank, and thus any such a group is semi-generalized Bassian} (see Theorem~\ref{65}). We finish our work in the final third section with certain comments and a list of several open problems which quite logically arise and, hopefully, stimulate a future study on the explored subject.

\section{Main Results and Proofs}

We start here with the following two constructions which generalize the given above two preliminary examples, thus motivating and enabling us to expand the existing in the subject results.

\begin{example}\label{(1)} Every divisible group is semi-generalized Bassian.
\end{example}

\begin{proof} This follows automatically, because for each divisible group $D$ any its subgroup is essential in
some direct summand of $D$ (cf. \cite[Theorem~24.4]{F1}).
\end{proof}

It is worthwhile noting that a torsion divisible group need {\it not} be generalized Bassian (and thus {\it not} Bassian), so we exhibit another type of a {\it proper} semi-generalized Bassian group.

\medskip

The next claim is a partial case of Theorem~\ref{6}, but for completeness of the exposition we give a new, independent verification of its validity.

\begin{example}\label{(2)} The group $G=\bigoplus_{\alpha}\mathbb{Z}(p^n)$ is semi-generalized Bassian for all ordinals $\alpha$ and integers $n\geq 1$.
\end{example}

\begin{proof} We claim any subgroup $H\leq G$ is essential in some direct summand of $G$. Indeed, write $H=\bigoplus_{i\in I}\langle x_i\rangle$ and let $x_i=p^{m_i}y_i$, where $m_i=\val{x_i}_G$. Then, $\langle y_i\rangle$,
$i\in I$, forms the direct sum $\bigoplus_{i\in I}\langle y_i\rangle=X$ and $H\leq X$. Moreover, assume that
$o(y_i)=p^l$, where $l<n$ for some $i\in I$, and $l\in \mathbb{Z}$. Thus, $$p^{l-1}y_i\in G[p]=p^{n-1}G[p].$$ So, $p^{l-1}y_i=p^{n-1}z$ for some $z\in G$, whence $y_i-p^{n-l}z\in G[p^{l-1}]$. However, because $l<n$, we deduce $G[p^{l-1}]\leq pG$, i.e., $y_i-p^{n-l}z\in pG$, $\val{y_i}_G>0$ and $\val{x_i}_G>m_i$, a contradiction. This means that $X$ is a direct sum of cyclic groups of order exactly $p^n$. We, therefore, infer that $X$ is a direct summand in $G$ (see \cite[Proposition~27.1]{F1}), as expected.
\end{proof}

The following technicality demonstrates that in the class of B+E-groups the concepts of being generalized Bassian and semi-generalized Bassian do coincide.

\begin{proposition}\label{1} Let $G$ be a {\rm {B+E}}-group. Then $G$ is generalized Bassian if, and only if, it is semi-generalized Bassian.
\end{proposition}

\begin{proof} Certainly, by definition, if $G$ is generalized Bassian, then it is semi-generalized-Bassian. As for the reverse implication, suppose that $G$ is semi-generalized Bassian. If the map $G\to G/N$ is injective for some $N\leq G$, then in view of \cite[Lemma 3.7]{DK}, the subgroup $N$ must be elementary and a direct summand of $T(G)$, the torsion subgroup of $G$. Furthermore, if we assume that $N$ is essential in some direct summand $A$ of $G$, it quickly follows that $A$ is contained in $T(G)$, so that $A$ must actually equal to $N$. This means that $N$ is, in turn, a direct summand of $G$, and hence $G$ is generalized Bassian, as wanted.
\end{proof}

We now continue with the following technical statement which is pivotal for proving the direct summand property of semi-generalized Bassian groups.

\begin{lemma}\label{summandlemma} Suppose $G$ is a mixed group with torsion $T$. Let $A$ be a summand of $G$ containing $N$ as an essential subgroup. If $G=H\oplus K$ is a decomposition with $N\subseteq H$ and $\pi:G\to H$ is the standard projection, then $\pi$ restricted to $A$ is injective, and $\pi(A)$  is a summand of $G$, and of $H$.
\end{lemma}

\begin{proof} The kernel of $\pi$ is $K$. Clearly, $K\cap N\subseteq K\cap H=\{0\}$, and since $N$ is essential in $A$, this implies that $K\cap A=\{0\}$. This means that $\pi$ is injective on $A$.  In addition, $\pi$ is obviously the identity on $N\subseteq H$.

Let $G=A\oplus C$. We claim that there is another decomposition $G=\pi(A)\oplus C$. Notice that this would also imply that $H=\pi(A)\oplus (C\cap H)$, completing the proof.

We first show that $\pi(A)\cap C=\{0\}$: So, suppose $x\ne 0$ is in this intersection. Let $x=\pi(a)$, where $0\ne a\in A$. It follows that there is an integer $n$ such that $0\ne na\in N$. Therefore, $0\ne na =\pi(na)=nx$ will be an element of $N\cap C\subseteq A\cap C=\{0\}$, which gives our desired contradiction.

So, we will be done if we can show that $G':=\pi(A)+C=G$:  Since $G=A\oplus C$, it will suffice to show that $A\subseteq G'$. Let $a\in A$ be arbitrary. Note that $a=(a-\pi(a))+\pi(a)$, and $\pi(a)$ is clearly in $G'$. It will, therefore, suffice to show that for every $a\in A$, $a-\pi(a)\in G'$. For any such $a\in A$, there is an integers $n$ such that $na\in N$. Since $n(a-\pi(a))=na-\pi(na)=na-na=0$, we can conclude that $a-\pi(a)\in T$. Consequently, it will suffice to show that $T\subseteq G'$.

Note that $T=T_A\oplus T_C=T_H\oplus T_K$ and $T_N$ will be essential in $T_A$. Therefore, to show $T\subseteq G'$, it suffices to assume $G$ and all of these groups are torsion. And from that, it is clear that all of these considerations can be handled one prime at a time, so we may assume that $G$ and all of these are $p$-groups, which we do.

Since $\pi(A)\cap C=\{0\}$, we can conclude that
$$
         G[p]=A[p]\oplus C[p]=N[p]\oplus C[p]=\pi(A)[p]\oplus C[p]=G'[p].
$$
And if $x\in G[p]$ equals $a+c$, for $a\in A[p]$, $c\in C[p]$, then
$$
      \val x_G=\min\{\val a_A, \val c_C\}\leq  \min\{\val a_{\pi(A)}, \val c_C\}= \val x_{G'}\leq \val x_G,
$$
which implies that $G'$ is pure in $G$. So, since they have the same socles, they must be equal, as required.
\end{proof}

We now arrive at our first chief result which expand the corresponding closeness of the direct summand property for both Bassian and generalized Bassian groups (compare with \cite{CDG1} and \cite{CDG2}).

\begin{theorem}\label{summands} The class of semi-generalized Bassian groups is closed under direct summands.
\end{theorem}

\begin{proof} Suppose $G$ is a semi-generalized Bassian group and $H$ is a summand of $G$; let $\pi:G\to H$ be the usual projection. Suppose $N$ is a subgroup of $H$ and $\phi:H\to H/N$ is an injective homomorphism. Extending $\phi$ to $G\to G/N$ by letting it equal the identity on a complementary summand of $H$, it remains an injective homomorphism. Therefore, $N$ is contained as an essential subgroup in some summand, say $A$, of $G$. An application of Lemma~\ref{summandlemma} guarantees that $\pi(A)$ will also be a summand of $H$ containing $\pi(N)=N$ as an essential subgroup, thereby proving the result.
\end{proof}

Our next result states as follows.

\begin{proposition}\label{4} A reduced torsion-free group $G$ is semi-generalized Bassian if, and only if, $G$ is Bassian or, equivalently, if, and only if, $G$ has finite rank.
\end{proposition}

\begin{proof} Assume the contrary that $G$ has infinite rank and that $M$ is a maximal independent subset of $G$. We, thus, have that $F=\langle M\rangle$ is a free group of infinite rank equal to $|G|$. Hence, there exists a subgroup
$X$ of $F$ with $F/X\cong D(G)$, where $D(G)$ is a divisible hull of $G$. So, $G$ then embeds in $G/X$. If we write $G=U\oplus V$, where $U/H$ is torsion, then $G/H=U/H\oplus V$ will not have a subgroup isomorphic to $D(G)$, as needed.
\end{proof}

As a consequence, we now continue by considering the behavior of the cartesian direct sum of a semi-generalized Bassian group without torsion.

\begin{corollary}\label{3} The following two statements are true:

(1) If $A$ is a torsion-free group and $\alpha$ is finite, then $G=\bigoplus_{\alpha} A$ is semi-generalized Bassian if, and only if, so does $A$.

(2) If $A$ is a torsion-free group and $\alpha$ is infinite, then $G=\bigoplus_\alpha A$ is semi-generalized Bassian if, and only if, $A$ is divisible.

\end{corollary}

\begin{proof} (1) is straightforward, so we omit the details.

(2) is a routine consequence of the next Theorem~\ref{5}, so we leave it for inspection by the interested reader.

\end{proof}

A concrete example of such a group $A$ is an arbitrary indecomposable torsion-free group of finite rank $>1$.

\medskip

We are now prepared to prove the following.


\begin{theorem}\label{5}
If $G$ is a torsion-free group, then $G$ is semi-generalized Bassian if, and only if, one of the following two conditions holds:

(1) $G$ has finite rank;

(2) $G\cong D\oplus R$, where $D$ is divisible (of infinite rank) and $R$ is a finite rank fully decomposable homogeneous group.
\end{theorem}

\begin{proof} "{\bf Necessity}". Applying a combination of Theorem~\ref{summands} and Proposition~\ref{4}, it follows at once that the reduced part $R$ has finite rank.

Let us now the rank of $D$ is infinite and choose $\{0\}\neq X\leq R$. Then, apparently, there exists an injection
$G\to G/X$ and so $X$ is essential in some direct summand of $R$, that is, the purification $X_*$ of $X$ is its direct summand. Therefore, in virtue of \cite[\S 86, Exersize 7]{F1}, the group $R$ is a fully decomposable homogeneous group (see also \cite{FKS}).

\medskip

"{\bf Sufficiency}". If $D$ has finite rank, then the group $G$ is Bassian. Assume now that the rank of $D$ is infinite, that $\{0\}\neq H\leq G$ and that there exists an injection $G\to G/H$. If $H\cap D=\{0\}$, then $H$
is contained in a reduced part $R$ of $G$ (note that $R$ is defined up to an isomorphism) and so $X_*$ is a direct summand of $R$. Assume also that $H'=H\cap D\neq\{0\}$, and suppose $H''$ be a $H'$-hight subgroup in $H$. It is well known that $H'\oplus H''$ is essential in $H$ (cf. \cite[\S 10, Exersize 6]{F1}). However, the subgroup $H''$
is contained in a reduced part $R$ of $G$. Consequently, a simple inspection shows that $H_*'\oplus H_*''$ is a direct summand of $G$ and, clearly, that $H\leq H_*'\oplus H_*''$ is an essential subgroup, as required.
\end{proof}

We shall now deal with $p$-groups. And so, we have now all the ingredients to establish the following surprising assertion.

\begin{theorem}\label{6} The reduced $p$-group $G$ has the property that every subgroup is an essential subgroup of a summand of $G$ if, and only if, for some positive integer $n$, we have
$$
G\cong \left (\bigoplus_I \mathbb Z_{p^n}\right) \oplus \left (\bigoplus_J \mathbb Z_{p^{n+1}}\right);
$$
that is, $f_G(\alpha)=0$ unless $\alpha=n-1,n$.
In particular, every such a group $G$ is semi-generalized Bassian.
\end{theorem}

\begin{proof} Suppose first that $G$ fails to be of this form. In other words, there are $n,m$ such that $n+1<m$ and $G$ has a summand of the form $Y\oplus Z $, where $Y\cong\mathbb Z_{p^n}$ and $Z\cong \mathbb Z_{p^m}$. Let $y\in Y$ have order $p$, $z\in Z$ have order $p^2$ and $x=y+z$. Consider the cyclic subgroup $N=\langle x\rangle$.  Since $m-2>(n+1)-2=n-1$, it follows that $x$ has height sequence $(n-1, m-1, \infty, \infty, \cdots)$. Note that this has two gaps. If $N$ were essential in a summand $S$, then $S$ would have to be cyclic. If $S=\langle b\rangle\cong \mathbb Z_{p^j}$, then the height sequence of $b$ would have exactly one gap, namely from $\val {p^{j-1} b}_S=j-1$ to $\val {p^j b}_S=\infty$, but this would imply that the height sequence of $x$ would also have exactly one gap, contrary to its construction. (Notice that we will see this argument several times in the sequel.)




\medskip

Conversely, suppose $G$ is of the above form.

\medskip

{\bf Claim:} For all $x\in G$, the inequality $px\ne 0$ forces that $\val {px}_G=\val x_G+1$. Since $px\ne 0$, we can conclude that  $k:=\val x_G\leq n-1$. Suppose, for a moment, that $k=n-1$. Since $\val {px}_G\geq n$ and $p^{n+1}G=0$, we can conclude that $\val {px}_G=n$. Suppose next that $k=\val x_G<n-1$ and $\val {px}_G>k+1$. Let $y\in p^{k+1}G$ satisfy $py=px$. Then, $x-y$ would be an element of $G[p]$ of height $k<n-1$, which cannot happen, thus substantiating our claim.

\medskip

Furthermore, turning to the initial proof itself, suppose $A$ is some subgroup of $G$. Let $B$ be a subgroup containing $A$ that is maximal with respect to the property $B[p]=A[p]$. Since $A$ is essential in $B$, it suffices to show that $B$ is a summand of $G$. To show $B$ is a summand, it suffice to show that it is pure in $G$. And to show it is pure, it suffices to show, for all non-zero $x\in B[p]$, that $k:=\val x_B=\val x_G$. Let $y\in B$ satisfy $p^k y=x$. If $\val y_G\geq 1$, then by the maximality of $B$ we must have $y\in B\cap pG=pB$. But, if $y=py'$ for some $y'\in B$, then $p^{k+1} y' =x$ would contradict that $\val x_B= k$. Therefore, $\val y_G=0=\val y_B$. And applying the Claim $k$ times, we can conclude that
$$\val x_B=k = \val {p^k y}_G=\val x_G,$$
as desired.

The second part of the statement is now immediate.
\end{proof}

We are now working towards a characterization of the $p$-groups that are semi-generalized Bassian. The following contains a key step in that way.




\begin{lemma}\label{simplify} Suppose $G$ is a $p$-group with a summand of the form $X=Y\oplus Z$ and $n$ is a positive integer. If either of the following two conditions holds, then $G$ is not semi-generalized Bassian.

(a) $Y\cong \mathbb Z_{p^n}$ and $Z$ is the direct sum of an infinite number of copies of $\mathbb Z_{p^m}$, where either $n+2\leq m<\omega$ or $m=\infty$.

(b) $Y$ is the direct sum of an infinite number of copies of $\mathbb Z_{p^n}$ and $Z\cong \mathbb Z_{p^\infty}$.
\end{lemma}

\begin{proof} Since by Theorem~\ref{summands} a summand of a semi-generalized Bassian group retains that property, we may assume $G=X$.

Regarding (a), let $Z=Z_1\oplus Z_2\oplus Z'$, where $Z_1\cong Z_2\cong \mathbb Z_{p^n}$ and $\sigma: Z\cong Z'$ is an isomorphism. Next, let $y\in Y$ have order $p$, $z\in Z_1$ have order $p^2$ and $x=y+z\in G$. If $N=\langle x\rangle$, then it follows that $N$ is not an essential subgroup of a summand $S$. Suppose otherwise: since $\val x_N=n-1\ne \infty$, $S$ will also be cyclic, but this contradicts the fact that $\val {px}_G=\val {pz}_G\ne n=\val x_G+1$, which cannot happen in a cyclic group.

We need to construct an injection $\phi:G\to G/N=\overline G$. Note that
$$G=Y\oplus Z_1\oplus Z_2\oplus Z'\ {\rm and \ } \overline G\cong((Y\oplus Z_1)/N)\oplus Z_2\oplus Z'.$$
There is clearly an injection $\tau:Y\to Z_2$. So,
$$\phi=(\tau,\sigma):G=Y\oplus Z\to Z_2\oplus Z'\subseteq \overline G$$
is the required injective homomorphism.

Turning to (b), again, let $Y=Y_1\oplus Y_2$, where $\tau:Y\cong Y_1$ is an isomorphism and $Y_2\cong \mathbb Z_{p^n}$. Let $y\in Y_2$ have order $p$, $z\in Z$ have order $p^2$ and $x=y+z$. As before, if $N=\langle x\rangle$, then $N$ will not be an essential subgroup of a summand of $G$. So, we need to show that there is an injection $\phi:G\to \overline G.$

Clearly, $(Y_2\oplus Z)/N$ is not reduced, so there is an injective homomorphism $\sigma:Z\to (Y_2\oplus Z)/N$. Therefore,
$$\phi=(\tau, \sigma):G=Y\oplus Z\to Y_1\oplus ((Y_2\oplus Z)/N)\cong \overline G$$
is our required injection.
\end{proof}

This brings us to our promised above characterization of semi-generalized Bassian $p$-groups.

\begin{theorem}\label{main} If $G$ is a $p$-group, then $G$ is semi-generalized Bassian if, and only if, one of the next three conditions hold:

(1) $G$ is divisible;

(2) $G$ has finite $p$-rank;

(3) For some positive integer $n$, $G\cong G'\oplus F$, where $G'$ is isomorphic to a group as in Theorem~\ref{6} and $F$ is a finite subgroup such that $F[p]\subseteq p^{n+1} G[p]$.
\end{theorem}

\begin{proof} We first verify that a group satisfying any of these three conditions is semi-generalized Bassian. First, a divisible group has the property that any subgroup is an essential subgroup of a summand; namely, any copy of the divisible hull of the subgroup. Therefore, any group satisfying (1) is semi-generalized Bassian.

In considering (2) and (3), we will assume $N$ is a subgroup of $G$, $\pi: G\to G/N:= \overline G$ is the canonical epimorphism and $\phi:G\to \overline G$ is an injective homomorphism; we need to show $N$ is an essential subgroup of a summand of $G$.

Suppose (2) holds, so $G$ has finite $p$-rank, i.e., $G=R\oplus D$, where $R$ is finite and $D$ is divisible of finite rank. We assert that $N\subseteq D$, which will complete the proof since, again, any subgroup of a divisible group is essential in a summand.

Applying \cite[Lemma~2.1(b)]{K1}, $\phi$ must be an isomorphism $G\to \overline G$; let $D'= \phi(D)$ be the maximal divisible subgroup of $ \overline G$. Clearly, $\pi$ determines a surjective homomorphism $\pi':G/D\to \overline G/D'$. And since $G/D\cong R\cong \phi(G)/\phi(D)=\overline G/D'$, and $R$ is finite, we can conclude that $\pi'$ is an isomorphism.

Evidently, $\pi'([N+D]/D) =0$, so that $[N+D]/D=0$, i.e., $N\subseteq D$, as desired.

\medskip

Lastly, suppose (3) holds, with $G\cong G'\oplus F$ as stated. We claim that $N\cap p^{n+1}G=0$. Clearly, $\phi$ and $\pi$ restrict to an injective and a surjective homomorphism $p^{n+1} G\to p^{n+1} \overline G$. But since $p^{n+1}G$ is finite, these are both isomorphisms. In particular, the claim follows.

Let $G''$ be any $p^{n+1}G$-high subgroup of $G$ containing $N$. It follows that $G\cong G''\oplus F$ and $G''\cong G'$, so that we may replace $G'$ by this new $G''$. So, we may assume that $G'$ contains $N$, and invoking Theorem~\ref{6}, $G'$ has a summand containing $N$ as an essential subgroup, so the same holds for $G$.

\medskip

Conversely, suppose $G$ is semi-generalized Bassian; we need to show it satisfies one of the above conditions. Let $G=R\oplus D$, where $R$ is reduced and $D$ is divisible.

We first assert that $R$ must be bounded, so assume otherwise. It follows that $R\cong B\oplus T$, where $B$ is bounded and $T$ has the same infinite rank and final rank. If we can show $T$ is not semi-generalized Bassian, then by Lemma~\ref{summand}, $G$ is not either, giving our contradiction.

Consulting with \cite[Lemma~2.3(a)]{K1}, there is a surjective homomorphism $\pi:T\to E$, where $E$ is a divisible hull of $T$. If $N$ is the kernel of $\pi$, we claim that there is no loss of generality in assuming that $N$ is not essential in $T$: If $T=C\oplus T'$, where $C$ is cyclic, then $E/\pi(T')$ must be simultaneously divisible and finite, hence $\{0\}$. Therefore, $\pi(T')=E$, so that we can possibly redefine $\pi$ to be injective on $C$.

Let $\phi:T\to E$ be the inclusion. So, if $T$ is semi-generalized Bassian, then $N$ is essential in a summand $S$ of $T$. If $T=S\oplus X$, then since $N$ is not essential in $T$, the summand $X$ is a non-zero reduced group. Since $N\subseteq S$, it follows that
$$
          X\cong T/S\cong (T/N)/(S/N)\cong E/\pi(S).
$$
This, however, is a contradiction, because $X\ne 0$ is reduced and $E/\pi(S)$ is divisible.

\medskip

Continuing with the whole proof, if $R=\{0\}$, then we are in case (1), and so we are done. Thus, we may assume $R$ is non-zero (and bounded).

Note that the application of Lemma~\ref{simplify}(a) assures that $D$ must have finite $p$-rank. Next, if $D\ne 0$, then since $R$ is bounded, Lemma~\ref{simplify}(b) insures that $R$ is finite, so that (2) holds.

Finally, suppose $D=\{0\}$. If $R$ has finite rank, then (2) holds; so assume it has infinite rank. Recall that $R$ is bounded, so we can let $n$ be the largest integer such that $f_R(n)$ is infinite. Consequently, Lemma~\ref{simplify}(a) gives that $f_R(j)=0$ whenever $j+2\leq n$. This clearly implies that $G$ satisfies (3), as desired.
\end{proof}

We now attack a major version of the mixed case, namely the splitting mixed groups. In this aspect, the next general statement is quite useful.

\begin{lemma}\label{newone1} If $G$ is a semi-generalized Bassian group, then any its $p$-component $G_p$ has a bounded reduced part, and so $G_p$ is a semi-generalized Bassian group which is a direct summand of $G$.
\end{lemma}

\begin{proof} Assume the contrary that the reduced part $C_p$ of $G_p$ is unbounded. Thus, we can represent $C_p$ as
$C_p=C_p'\oplus C_p''$, where $C_p'$ is bounded and $r(C_p'')=\mathrm{final}\, r(C_p'')$ (see, for instance, \cite[\S 36, Exercise 11]{F1}). However, in $C_p''$ there exists a basic subgroup $B_p$ such that $r(C_p'')=r(C_p''/B_p)$ (see
\cite[Theorem 35.6]{F1}), so that there exists an injection $f: C_p''\to C_p''/B_p$. Since $C_p''/B_p$ is divisible, we can extend $f$ to a homomorphism $\overline{f}: G\to C_p''/B_p$. Moreover, $G/B_p=C_p''/B_p\oplus K/B_p$ for some
$K\leq G$. 

Let us now $\varphi: G\to K/B_p$ be an epimorphism and set $\psi:=\varphi+\overline{f}$. Assume $0\neq g\in\ker\psi$.
If, for a moment, $\varphi(g)=0$, then $g\in C''_p$ and so $\overline{f}(g)\neq 0$ contradicting $g\in\ker\psi$. Now, we can choose $0\neq\varphi(g)$ and thus $$0\neq\overline{f}(g)=-\varphi(g)\in C_p''/B_p\cap K/B_p=\{0\},$$ which also is a contradiction. Consequently, $\ker\psi=\{0\}$ and $\psi: G\to G/B_p$ is an injection. By assumption, $B_p$ is essential in some direct summand, say $U$ of $G$, whence $G=U\oplus V$. Furthermore, since $B_p$ essential in $U$, one sees that $U\leq G_p$, $G_p=U\oplus (G_p\cap V)$ and $B_p[p]=U[p]$. Therefore, \cite[Lemma 66,1]{F1} allows us to deduce that $B_p$ is a direct summand of $G_p$, that is impossible. Consequently, the reduced part of $G_p$ is bounded, as asked for. Finally, Theorem~\ref{summands} ensures that $G_p$ is a semi-generalized Bassian group as being a direct summand of $G$, as promised.
\end{proof}

\begin{proposition}\label{newone2} The splitting group $G=T\oplus R$, where $R$ is the torsion-free part of $G$, is a semi-generalized Bassian group if, and only if, $T$ and $R$ are both semi-generalized Bassian groups and, moreover, if the rank of $R$ is infinite, then $T$ is divisible.
\end{proposition}

\begin{proof} "\textbf{Necessity}". Referring to Theorem~\ref{summands}, both $T$ and $R$ are semi-generalized Bassian groups. Let us assume that the rank of $R$ is infinite. With Theorem~\ref{5} at hand, one writes that $R=D\oplus K$, where $D$ is divisible and $K$ is a finite rank fully decomposable homogeneous group. Supposing that $T_p$ is non-divisible for some prime $p$, then in $T$ there will exist a cyclic direct summand $\langle a\rangle$ with $o(a)=p^n$. Besides, letting $0\neq b\in R$ and $X=\langle a+pb\rangle$, since $X$ is torsion-free, we derive that $G/X$ contains a subgroup isomorphic to $T$, and since $D$ is divisible of infinite rank, $G/X$ contains a subgroup isomorphic to $R$. So, there is an injection $G\to G/X$ and, consequently, $\langle x\rangle$ is contained in some direct summand of $G$ of rank $1$. So, one infers that $|x|_G = 0$ and $|p^nx|_G\geq n+1$, that is manifestly imposable in a torsion-free group, as expected.

"\textbf{Sufficiency}". Let $X\leq G$ and there exists an injection $G\to G/X$. If we assume, for a moment, that the rank of $R$ is finite, then $X\leq T$ and the result is obviously true, because $T$ is a semi-generalized Bassian group.

Suppose now that the rank of $R$ is infinite. Thus, $T$ is divisible. If $X\cap T=\{0\}$, then as any divisible subgroup is an absolute direct summand, $X$ is contained in some torsion-free part of $G$, and so we are done.

If, however, $X\cap T\neq\{0\}$, then $X\cap T$ is contained as an essential subgroup in some direct summand, say $T'$, of $T$ and we may write $G=T'\oplus B$. Letting $\pi: G\to T'$ be the corresponding projection, then it follows that $$X\leq \pi X\oplus (1-\pi)X=\overline{X}.$$ Since $X\cap T$ is essential in $T'$, then one verifies that $X\cap T$ is essential in $\pi X$. Thus, the subgroup $(1-\pi)X$ is torsion-free and, since $\pi X$ is torsion, then $X\cap B$ is essential in $(1-\pi)X$, so that $X$ is essential in $\overline{X}$. We now have
$$(1-\pi)X\leq T''\oplus B',$$ where $T''=T(B)$. Since $$((1-\pi)X)\cap T''=\{0\}$$ and $T''$ is divisible, thus is an absolute direct summand, it is then possible to write that $(1-\pi)X\leq B'$. So, $(1-\pi)X$ is essential in some direct summand, say $B''$, in $B'$. Consequently, $X\leq T''\oplus B''$ and $X$ is essential in the direct summand
$T''\oplus B''$ of $G$, as required.
\end{proof}

We now need to prove the following two technical claims (for the first one, see \cite{FKS} too) in order to start the description of all groups with the property that each subgroup is essential in a direct summand of a given group.

\begin{lemma}\label{61} In a $p$-group $G$ every pure subgroup is a direct summand if, and only if, $G=D\oplus R$, where $D$ is a divisible group and $R$ is a bounded group.
\end{lemma}

\begin{proof} "\textbf{Necessity}". It follows from the obvious fact that in a non-bounded $p$-group its basic subgroup is {\it not} a direct summand.

\medskip

"\textbf{Sufficiency}". Choose $H\leq G$ is pure in $G$ and write $H=D'\oplus R'$, where $D'$ is the divisible part
of $H$. Thus, $D'$ is a direct summand in $D$ and since $R'\cap D=\{0\}$ we know that $R'$ is contained in some
reduced part of $G$, where $R'$ is a direct summand. So, one verifies that $H$ is a direct summand in $G$, as required.
\end{proof}

\begin{lemma}\label{62} If $G$ is a mixed group in which every subgroup is essential in some direct summand of $G$, then the torsion part of $G$ is divisible.
\end{lemma}

\begin{proof} If we assume the contrary that some of the $p$-component, say $G_p$, is not divisible, then it will have a cyclic direct summand $\langle a\rangle$ with ${\rm order}(a)=p^n$, which is a direct summand in $G$ as well, writing $G=\langle a\rangle\oplus B$. Suppose $x=a+pb$, where $b\in B$, $o(b)=\infty$. Then, if $\langle x\rangle$ is contained in some direct summand $X$ of $G$ of rank $1$, we derive that $\val x_G=0$ and $\val {p^nx}_G\geq n+1$, that is imposable in a torsion-free group, as expected.
\end{proof}

As a direct consequence, we yield the following.

\begin{corollary}\label{63} If $G$ is a mixed group in which every subgroup is essential in some direct summand of $G$, then $G$ has the following shape $G=D\oplus R$, where $D$ is a divisible group and $R$ is a reduced torsion-free group.
\end{corollary}

We now have all the machinery to menage the proof of the following statement which, along with Theorem~\ref{6} in hand, insures a complete description of groups in which every subgroup is essential in a direct summand of the whole group.

\begin{theorem}\label{65} If $G$ is either a mixed or a torsion-free group, then every subgroup of $G$ is essential in some direct summand if, and only if, $G=D\oplus R$, where $D=T(D)\oplus D_0$ is a divisible group and $R$ is a homogeneous fully decomposable group of finite rank. In particular, each such a group $G$ is semi-generalized Bassian.
\end{theorem}

\begin{proof} "\textbf{Necessity}". It follows immediately from Lemma~\ref{62} and \cite[\S86, Exersize 7]{F1}.

\medskip

"\textbf{Sufficiency}". Put $X\leq G$. If we assume for a moment that $X\cap T(D)=\{0\}$, then as a divisible group is always an absolute direct summand (see, for instance, \cite[\S 9, Exersize 8]{F1}), it is possible to choose $D_0\oplus R$ such that $X\leq D_0\oplus R$, and so the assertion is true since the latter group is semi-generalized Bassian.

Let us now $T(X)\neq\{0\}$. Look at the subgroup $\overline{X}=\overline{T(X)}\oplus X_0$, where $\overline{T(X)}$
is the divisible hull of $T(X)$ in $T(D)$. Thus, it is clear that $T(X_0)=\{0\}$ and that $X$ is essential in $\overline{X}$. Likewise, we may choose the direct sum $D_0\oplus R$ so that $X_0\leq D_0\oplus R$. Next, letting $U$
be a direct summand in $D_0\oplus R$ in which $X_0$ is essential, we infer that $\overline{T(X)}\oplus U$ is a direct summand in $G$, whence $X$ is essential in it, as required.
\end{proof}

\section{Concluding Discussion and Open Questions}

We begin this section with the following additional comments.

\begin{remark}\label{9} Unfortunately, we were unable to characterize in a complete form all mixed semi-generalized Bassian groups, although in some partial cases we have done that (see, for more details, Theorem~\ref{65}).
\end{remark}

So, in closing, we pose certain challenging questions of some interest and importance which directly arise.

\medskip

The first question is motivated by the received above characterization results in the mixed case.

\medskip

\noindent{\bf Problem 1.} Is it true that a group $G$ is (semi-)generalized Bassian if, and only if, both $T(G)$ and $G/T(G)$ are so?

\medskip

If {\bf yes}, this will substantially help us to describe in some direction the structure of all mixed (semi-)generalized Bassian groups, because we already have characterized above both the subgroup $T(G)$ and the quotient-group $G/T(G)$ (compare with Proposition~\ref{newone2} as well).

\medskip

The second final question is subsumed by Theorem~\ref{main}.

\medskip

\noindent{\bf Problem 2.} Is any semi-generalized Bassian $p$-group of the sort $G\oplus B$, where $G$ is a generalized Bassian $p$-group and $B$ is a bounded $p$-group such that $f_B(\alpha)=0$ unless either $\alpha=n-1$ or $\alpha=n$?

\medskip


\medskip
\medskip

\noindent {\bf Funding:} The work of the first-named author A.R. Chekhlov was supported by the Ministry of Science and Higher Education of Russia (agreement No. 075-02-2023-943). The work of the second-named author P.V. Danchev was partially supported by the Bulgarian National Science Fund under Grant KP-06 No. 32/1 of December 07, 2019, as well as by the Junta de Andaluc\'ia under Grant FQM 264, and by the BIDEB 2221 of T\"UB\'ITAK.

\end{document}